\documentclass[12pt]{amsart}
\usepackage{amsmath}
\usepackage{amssymb}
\newtheorem{thm}{Theorem}[section]
\newtheorem{cor}[thm]{Corollary}
\newtheorem{lem}[thm]{Lemma}

\theoremstyle{definition}
\newtheorem{defn}[thm]{Definition}
\theoremstyle{remark}
\newtheorem{rem}[thm]{\bf{Remark}}
\numberwithin{equation}{section}

\newcommand{\beas}{\begin{eqnarray*}}
\newcommand{\eeas}{\end{eqnarray*}}
\newcommand{\bes} {\begin{equation*}}
\newcommand{\ees} {\end{equation*}}
\newcommand{\be} {\begin{equation}}
\newcommand{\ee} {\end{equation}}
\newcommand{\bea} {\begin{eqnarray}}
\newcommand{\eea} {\end{eqnarray}}

\newcommand{\R}{\mathbb R}

\newcommand{\C}{\mathbb C}

\newcommand{\N}{\mathbb N}

\begin{document}

\title[Parabolic convergence] {On parabolic convergence of positive solutions of the heat equation}
\author[J. Sarkar]{Jayanta Sarkar}
\address{Stat Math Unit, Indian Statistical Institute, 203 B. T. Road, Calcutta 700108}
\email{jayantasarkarmath@gmail.com}
\thanks{The author is supported by a research fellowship from Indian Statistical Institute.}
\keywords{Heat equation, Fatou-type theorems, Parabolic convergence, Derivative of measures.}
\subjclass[2010]{31A20, 31B25, 42B99}
\begin{abstract}
In this article, we study certain type of boundary behaviour of positive solutions of the heat equation on the upper half-space of $\R^{n+1}$. We prove that the existence of the parabolic limit of a positive solution of the heat equation at a point in the boundary is equivalent to the existence of the strong derivative of the boundary measure of the solution at that point. Moreover, the parabolic limit and the strong derivative are equal.
\end{abstract}
\maketitle
\section{Introduction}
We consider the heat equation
\begin{equation*}
\Delta u(x,t)=\frac{\partial}{\partial t}u(x,t),
\end{equation*}
on the upper half space $\R^{n+1}_+=\{(x,t)\mid x\in\R^n,t>0\}$ where $\Delta=\sum_{i=1}^{n}\frac{\partial^2}{\partial x_i^2}$ is the Laplace operator on $\R^n$.  The fundamental solution of the heat equation is the Gauss-Weierstrass kernel or the heat kernel of $\R^{n+1}_+$ and is given by
\begin{equation*}
W(x,t)=(4\pi t)^{-\frac{n}{2}}e^{-\frac{\|x\|^2}{4t}},\:(x,t)\in\R^{n+1}_+.
\end{equation*}
In this article, by a measure $\mu$ we will always mean a complex Borel measure or a signed Borel measure such that the total variation $|\mu|$ is locally finite, that is, $|\mu|(K)$ is finite for all compact sets $K$. If $\mu(E)$ is nonnegative for all Borel measurable sets $E$ then $\mu$ will be called a positive measure. The Gauss-Weierstrass integral of a measure $\mu$ on $\R^n$ is given by the convolution
\begin{equation*}
W\mu(x,t)=\int_{\R^n}W(x-y,t)\:d\mu(y),\:\:\:\: x\in\R^n,\:\: t\in (0,\infty ),
\end{equation*}
whenever the above integral exists. For measures on $\R^n$ with well defined Gauss-Weierstrass integral the following Fatou-type result is well-known \cite[Theorem 3]{W}: If for $x_0\in\R^n$
\begin{equation}\label{f1}
D_{sym}\mu (x_0):=\lim_{r\to 0}\frac{\mu (B(x_0,r))}{m(B(x_0,r))}=L,
\end{equation}
then
\begin{equation}\label{f2}
\lim_{t\to 0}W\mu (x_0,t)=L,
\end{equation}
where $m$ denotes the $n$-dimensional Lebesgue measure and $B(x_0,r)=\{x\in\R^n: \|x-x_0\|<r\}$ is the open ball centred at $x_0$ and radius $r>0$. It is also known that  the converse implication is false in general (see \cite[Theorem 3]{G}). However, in \cite{G, W}, it was proved that the  existence of the limit (\ref{f2}) does imply (\ref{f1}) if $\mu$ is positive. This was motivated by an earlier work of Loomis \cite{L} (see also \cite{Ru}) regarding converse of Fatou theorem for Poisson integral of positive measures. It then becomes natural to enquire for versions of the Fatou theorem and its converse for more general notion of convergence, namely the parabolic convergence (see Definition \ref{impdefn}, $i)$), of the Gauss-Weierstrass integrals of measures to boundary values. For $n=1$, these questions were answered by Gehring \cite{G}. To explain Gehring's results we will need some notation. For $x_0\in\R^n$ and $\alpha>0$, we define the parabolic domain $\texttt{P}(x_0,\alpha)\subset \R^{n+1}_+$ with vertex $x_0$ and aperture $\alpha$ by
\begin{equation}\label{pdomain}
\texttt{P}(x_0,\alpha)=\{(x,t)\in\R^{n+1}_+:\|x-x_0\|^2<\alpha t\}.
\end{equation}
By a positive solution of the heat equation, we shall always mean a solution of the heat equation that is nonnegative. Gehring's result uses the following characterization of positive solutions of the heat equation due to Widder \cite[Theorem 6]{Wi}: if $u$ is a positive solution of the heat equation on $\R^2_+$ then there exists a nondecreasing function $\beta$ defined on $\R$ such that
\begin{equation*}
u(x,t)=\int_{\R}W(x-y,t)\:d\beta (y),\:\:\:\:x\in\R,\:\: t>0.
\end{equation*}
We note that monotonicity of the function $\beta$ implies almost everywhere existence of its derivative $\beta'$. 
The following are the results of Gehring \cite[Theorem 2, Theorem 5]{G} whose higher dimensional generalization is the main topic of this paper.
\begin{thm}
\label{g}
Suppose $u$ is a positive solution of the heat equation on $\R^2_+$, $x_0\in\R$ and $\beta$ as above.
\begin{enumerate}
\item If $\beta'(x_0)=L$ then for each $\alpha>0$,
\begin{equation*}
\lim_{\substack{(x,t)\to(x_0,0)\\(x,t)\in \texttt{P}(x_0,\alpha)}}u(x,t)=L.
\end{equation*}
\item If for two distinct real numbers $\alpha_1$, $\alpha_2$
\begin{equation*}
\lim_{t\to 0}u(x_0+\alpha_1\sqrt{t},t)=L=\lim_{t\to 0}u(x_0+\alpha_2\sqrt{t},t),
\end{equation*}
then $\beta'(x_0)=L$.
\end{enumerate}
\end{thm}
A remarkable consequence of this result is the fact that if a positive solution $u$ of the heat equation converges to the same limiting value along two different parabolic paths through $(x_0,0)$ then it converges to the same value through all $\texttt{P}(x_0,\alpha)$. 

It is known that positive solutions of the heat equation on $\R^{n+1}_+$, $n\geq 1$, are given by Gauss-Weierstrass integrals of positive measures defined on $\R^n$. However, it is not clear how to interpret the derivative $\beta'$ (appearing in the theorem above) in higher dimension. In this paper, we show that a possible approach to solve this problem is to consider the so called strong derivative of measures (Definition \ref{impdefn}, $ii)$). The main tool used in the proof of  Theorem \ref{g} is Wiener's Tauberian theorem. But it is not at all clear at the moment whether the same approach can be adapted to prove the corresponding result for $\R^n$, $n>1$. The proof of our main result, Theorem \ref{main}, uses a completely different technique which can be traced back to a work of Ramey and Ullrich \cite{UR} on positive harmonic functions on $\R^{n+1}_+$. It is worth pointing out that a recent result of Bar \cite{B} on generalization of Montel's theorem (see Lemma \ref{montel}) plays a crucial role in the proof of our main theorem. It is this result which prohibits us from extending converse of Fatou theorem  for more general approximate identities other than Poisson kernel or heat kernel.

This paper is organised as follows: In the next section, we state and prove all the results needed to prove the main theorem, Theorem \ref{main}. The statement and proof of this theorem is given in the last section.

\section{preliminaries}
We start with the following simple result regarding Gauss-Weierstrass integral of measures (see \cite[Theorem 4.4]{W1}).
\begin{lem}\label{wat1}
Suppose $\mu$ is a measure on $\R^n$ such that $W\mu(x_0,t_0)$ is finite at some point $(x_0,t_0)\in\R^{n+1}_+$. Then $W\mu$ is well defined and is a solution of the heat equation in $\{(x,t)\mid x\in\R^n, t\in (0,t_0)\}$.
\end{lem}
Let $M$ denote the set of all positive measure $\mu$ on $\R^n$ whose Gauss-Weierstrass integral exists on  $\R^{n+1}_+$. In view of the previous lemma, we have
\begin{equation*}
M=\{\mu\:\text{is a positive measure on}\:\R^n:W\mu(0,t)\text{ is finite for all $t\in (0,\infty)$}\}.
\end{equation*}
For $(x,t)\in\R^{n+1}_+$, we set
\begin{equation*}
h(x)=W(x,1)=(4\pi)^{-n/2}e^{-\|x\|^2/4}.
\end{equation*}
Then we have
\begin{equation*}
W(x,t)=t^{-\frac{n}{2}}h\left(\frac{x}{\sqrt{t}}\right)=h_{\sqrt{t}}(x),\:\:\:\:(x,t)\in\R^{n+1}_+.
\end{equation*}
For a measurable function $f$ defined on $\R^n$, its Gauss-Weierstrass integral is given by
\begin{equation*}
Wf(x,t)=f*h_{\sqrt{t}}(x),\:\:\:\:(x,t)\in\R^{n+1}_+,
\end{equation*}
whenever the above convolution makes sense. From now onwards, whenever an integral is involved, we will write $dx$ instead of $dm(x)$.
It is well-known that if $f\in C_c(\R^n)$ then $Wf(\cdot ,t)$ converges to $f$ uniformly as $t$ goes to zero. However, a stronger result is true.
\begin{lem}\label{unif}
If $f\in C_c(\R^n)$ then
\begin{equation*}
\lim_{t\to 0}\frac{Wf(\cdot,t)}{h}=\frac{f}{h},
\end{equation*}
uniformly on $\R^n$.
\end{lem}
\begin{proof}
We assume that $\text{supp}f\subset B(0,R)$ for some $R>0$ . Since $h$ is bounded below by a positive number on $B(0,2R)$,
\begin{equation*}
\lim_{t\to 0}\frac{Wf(x,t)}{h(x)}=\frac{f(x)}{h(x)},
\end{equation*}
uniformly for $x\in B(0,2R)$. Hence, it suffices to prove that
\begin{equation*}
\lim_{t\to 0}\frac{Wf(\cdot,t)}{h}=0
\end{equation*}
uniformly for $x\in \R^n\setminus B(0,2R)$.
We write
\begin{equation}\label{l1e1}
\frac{Wf(x,t)}{h(x)}=\frac{1}{h(x)t^{\frac{n}{2}}}\int_{B(0,R)}f(y)h\left(\frac{x-y}{\sqrt{t}}\right)\:dy,\:\:x\in \R^n\setminus B(0,2R).
\end{equation}
For $\|x\|\geq2R$ and $\|y\|<R$, it follows from the triangle inequality that
\begin{equation}\label{revtri}
\|x-y\|\geq \|x\|-\|y\|>\frac{\|x\|}{2}.
\end{equation}
Since $h$ is radially decreasing and $\text{supp}f\subset B(0,R)$, we get from (\ref{l1e1}) that
\begin{eqnarray}
\left|\frac{Wf(x,t)}{h(x)}\right|&\leq& \frac{1}{h(x)t^{\frac{n}{2}}}\int_{B(0,R)}|f(y)|h\left(\frac{x-y}{\sqrt{t}}\right)dy\nonumber\\
&\leq& \frac{1}{h(x)t^{\frac{n}{2}}}\int_{B(0,R)}|f(y)|h\left(\frac{x}{2\sqrt{t}}\right)dy\nonumber\\
&=&\frac{h\left(\frac{x}{2\sqrt{t}}\right)}{h(x)t^{\frac{n}{2}}}\|f\|_{L^1(\R^n)}.\nonumber
\end{eqnarray}
Using the expression of the function $h$ it follows that for $\|x\|\geq2R$ and $t\in (0,1/4)$
\begin{equation*}
\frac{h\left(\frac{x}{2\sqrt{t}}\right)}{h(x)t^{\frac{n}{2}}}= t^{-\frac{n}{2}}e^{-\frac{\|x\|^2}{4}(\frac{1}{4t}-1)}
\leq  t^{-\frac{n}{2}}e^{-R^2(\frac{1}{4t}-1)}
\leq\frac{n!4^n}{R^{2n}}\frac{t^{\frac{n}{2}}}{(1-4t)^n}.
\end{equation*}
This proves the result.
\end{proof}
\begin{lem}\label{fubini-heat}
If $\nu\in M$ and $f\in C_c(\R^n)$ then for each fixed $t>0$,
\begin{equation*}
\int_{\R^n}Wf(x,t)\:d\nu(x)=\int_{\R^n}W\nu(x,t)f(x)\:dx.
\end{equation*}
\end{lem}
\begin{proof}
If $\nu$ is finite then the result can be easily proved for $f\in L^1(\R^n)\cap L^{\infty}(\R^n)$. For a general $\nu\in M$, the result follows by interchanging integrals provided
\begin{equation*}
\int_{\R^n}\int_{supp f}|f(y)|W(x-y,t)\:dy\:d\nu(x)<\infty.
\end{equation*}
We assume that $supp f\subset B(0,R)$ for some $R>0$.
Now,
\begin{eqnarray*}
&&\int_{\R^n}\int_{B(0,R)}|f(y)|W(x-y,t)\:dy\:d\nu(x)\\
&=&(4\pi t)^{-\frac{n}{2}}\left(\int_{B(0,2R)}\int_{B(0,R)}e^{-\frac{\|x-y\|^2}{4t}}|f(y)|\:dy\:d\nu(x)+\int_{B(0,2R)^c}\int_{B(0,R)}e^{-\frac{\|x-y\|^2}{4t}}|f(y)|\:dy\:d\nu(x)\right)\\
&\leq & (4\pi t)^{-\frac{n}{2}}\left(\int_{B(0,2R)}\int_{B(0,R)}|f(y)|\:dy\:d\nu(x)+\int_{B(0,2R)^c}\int_{B(0,R)}e^{-\frac{\|x\|^2}{16t}}|f(y)|\:dy\:d\nu(x)\right)\\
&&\:\:\:\:\:\:\:\:\:\:\:\:\:\:\:\:\:\:\:\:\:\:\:\:\:\:\:\:\:\:\:\:\:\:\:\:\:\:\:(\text{by the inequality (\ref{revtri})})\\
&\leq & (4\pi t)^{-\frac{n}{2}}\nu(B(0,2R))\|f\|_{L^1(\R^n)}+4^{\frac{n}{2}}\|f\|_{L^1(\R^n)}W\nu(0,4t).\end{eqnarray*}
The last quantity is finite as $\nu\in M$. This completes the proof.
\end{proof}
The following notions will be used throughout the paper.
\begin{defn}\label{impdefn}
\begin{enumerate}
\item[i)] A function $u$ defined on $\R^{n+1}_+$ is said to have parabolic limit $L\in\C$, at $x_0\in\R^n$ if for each $\alpha>0$
\begin{equation*}
\lim_{\substack{(x,t)\to(x_0,0)\\(x,t)\in \texttt{P}(x_0,\alpha)}}u(x,t)=L,
\end{equation*}
where $\texttt{P}(x_0,\alpha)$ is as defined in (\ref{pdomain}).
\item[ii)] Given a measure $\mu$ on $\R^n$, we say that $\mu$ has strong derivative $L\in\C$ at $x_0$ if
\begin{equation*}
\lim_{r\to 0}\frac{\mu(x_0+rB)}{m(rB)}=L
\end{equation*}
holds for every open ball $B\subset\R^n$. Here, $rE=\{rx\mid x\in E\}$, $r>0$, $E\subset\R^n$. The strong derivative of $\mu$ at $x_0$, if it exists, is denoted by $D\mu(x_0)$.
\item[iii)] A sequence of functions $\{u_j\}$ defined on $\R^{n+1}_+$ is said to converge normally to a function $u$ if $\{u_j\}$ converges to $u$ uniformly on compact subsets of $\R^{n+1}_+$.
\item[iv)] A sequence of functions $\{u_j\}$ defined on $\R^{n+1}_+$  is said to be locally bounded if given any compact set $K\subset\R^{n+1}_+$, there exists a positive constant $C_K$ such that
for all $j$ and all $x\in K$
\begin{equation*}
|u_j(x)|\leq C_K.
\end{equation*}
\item[v)] A sequence $\{\mu_j\}$ of positive measures on $\R^n$ is said to converge to a positive measure $\mu$ on $\R^n$ in weak* if
\begin{equation*}
\lim_{j\to\infty}\int_{\R^n}\phi(y)\:d\mu_j(y)=\int_{\R^n}\phi(y)\:d\mu(y),
\end{equation*}
for all $\phi\in C_c(\R^n)$.
\end{enumerate}
\end{defn}
\begin{rem}
\begin{enumerate}
\item It is clear from the definition above that if $D\mu (x_0)=L$ then $D_{sym}\mu (x_0)=L$. However, the converse is not true and can be seen from the following elementary example. Consider the measure $d\mu=\chi_{[0,1]}\:dm$ on $\R$. Then
\begin{equation*}
D_{sym}\mu (0)=\lim_{h\to 0}\frac{\mu \left((-h,h)\right)}{m\left((-h,h)\right)}=\lim_{h\to 0}\frac{1}{2h}\int_0^h dx=\frac{1}{2}.
\end{equation*}
However, the strong derivative of $\mu$ at zero does not exist. To see this, consider an interval of the form $I_1=(x-t,x+t)$ with $0<t<x$. Then for all positive $r$ smaller than $1/(x+t)$ we see that $rI_1$ is a subset of $[0,1]$ and hence
\begin{equation*}
\lim_{r\to 0}\frac{\mu (rI_1)}{m(rI_1)}=1.
\end{equation*}
On the other hand, if we choose $I_2=(x-t,x+t)$ with $x<0$ and $0<t<-x$, then for all $r>0$, $rI_2$ and $[0,1]$ are disjoint. Hence
\begin{equation*}
\lim_{r\to 0}\frac{\mu (rI_2)}{m(rI_2)}=0.
\end{equation*}
It follows that $D\mu (0)$ does not exist.
\item It is important to note that the results of Gehring, alluded to in the introduction, can also be stated in terms of the strong derivative of measures. The crux of the matter is the following: if $\beta:\R\to\R$ is a nondecresing, left continuous function  giving rise to a measure $\mu_{\beta}$ with $\mu_{\beta}([a,b))=\beta (b)-\beta (a)$, then $\beta'(x_0)=A$ if and only if $D\mu_{\beta}(x_0)=A$. Indeed, if $D\mu_{\beta}(x_0)=A$, then for every interval of the form $(a-s,a+s)$ with $s>0$ we have
\begin{eqnarray*}
A&=&\lim_{r\to 0+}\frac{\mu_{\beta}\left((x_0+ra-rs,x_0+ra+rs)\right)}{2rs}\\
&=& \frac{\beta (x_0+ra+rs)-\beta(x_0+ra-rs)}{2rs}.
\end{eqnarray*}
Now, by choosing $a=s=1/2$ (a one dimensional speciality) we get
\begin{equation*}
A=\lim_{r\to 0+}\frac{\beta (x_0+r)-\beta (x_0)}{r},
\end{equation*}
that is, the right hand derivative of $\beta$ at $x_0$ is $A$. By choosing $a=-1/2$ and $s=1/2$ we get that left hand derivative of $\beta$ at $x_0$ is also $A$. Conversely, if $\beta'(x_0)=A$ then for any interval of the form $I=(a-s,a+s)$ we have
\begin{eqnarray*}
\lim_{r\to 0}\frac{\mu_{\beta} (x_0+rI)}{2rs}&=&\lim_{r\to 0}\frac{\beta(x_0+ra+rs)-\beta (x_0+ra-rs)}{2rs}\\
&=&\lim_{r\to 0}\left(\frac{\beta (x_0+r(a+s))-\beta (x_0)}{r(a+s)}\times \frac{a+s}{2s}\right.\\
&& \left. -\frac{\beta (x_0+r(a-s))-\beta (x_0)}{r(a-s)}\times \frac{a-s}{2s}\right)\\
&=& A.
\end{eqnarray*}
That is, $D\mu_{\beta}(x_0)=A$.
\end{enumerate}
\end{rem}

\begin{lem}\label{normal}
Suppose $\{\mu_j\mid j\in\N\}\subset M$ and $\mu\in M$. If $\{W\mu_j\}$ converges normally to $W\mu$ then $\{\mu_j\}$ converges to $\mu$ in weak*.
\end{lem}
\begin{proof}
Let $f\in C_c(\R^n)$ with $supp f\subset B(0,R)$ for some $R>0$. For any $t>0$, we write
\begin{eqnarray}
&&\int_{\R^n}f(x)\:d\mu_j(x)-\int_{\R^n}f(x)\:d\mu(x)\nonumber\\
&=&\int_{\R^n}(f(x)-Wf(x,t))\:d\mu_j(x)+\int_{\R^n}Wf(x,t)\:d\mu_j(x)-\int_{\R^n}Wf(x,t)\:d\mu(x)\nonumber\\
&&\:\:\:\:\:\:\:\:\:+\int_{\R^n}(Wf(x,t)-f(x))\:d\mu(x).\label{ineq1}
\end{eqnarray}
Given $\epsilon>0$, by Lemma \ref{unif} we get some $t_0>0$, such that for all $x\in\R^n$
\begin{equation}\label{p1e4}
\frac{|Wf(x,t_0)-f(x)|}{h(x)}<\epsilon.
\end{equation}
Using Lemma \ref{fubini-heat} it follows from (\ref{ineq1}) that
\begin{eqnarray*}
&&\left|\int_{\R^n}f(x)\:d\mu_j(x)-\int_{\R^n}f(x)\:d\mu(x)\right|\\
&\leq &\int_{\R^n}|f(x)-Wf(x,t_0)|\:d\mu_j(x)+\int_{\R^n}|W\mu_j(x,t_0)-W\mu(x,t_0)||f(x)|\:dx\\
&&\:\:\:\:\:\:+\int_{\R^n}|Wf(x,t_0)-f(x)|\:d\mu(x)\\
&=&I_1+I_2+I_3.
\end{eqnarray*}
Applying (\ref{p1e4}), it follows that
\begin{equation*}
I_1=\int_{\R^n}\frac{|Wf(x,t_0)-f(x)|}{h(x)}h(x)\:d\mu_j(x)\leq \epsilon\int_{\R^n}h(x)\:d\mu_j(x)=\epsilon W\mu_j(0,1),
\end{equation*}
for all $j\in\N$. By the same argument we also have \begin{equation*}
I_3\leq\epsilon W\mu(0,1).
\end{equation*}
Since $\{W\mu_j\}$ converges to $W\mu$ normally, the sequence $\{W\mu_j(0,1)\}$, in particular, is bounded. Hence, taking $A$ to be the supremum of $\{W\mu_j(0,1)+W\mu(0,1)\}$, we get that
\begin{equation*}
I_1+I_3\leq2A\epsilon.
\end{equation*}
By normal convergence of $\{W\mu_j\}$ to $W\mu$, it follows that there exists $j_0\in\N$ such that for all $j\geq j_0$,
\begin{equation*}
\|W\mu_j-W\mu\|_{L^{\infty}(\overline{B(0,R)}\times\{t_0\})}<\epsilon.
\end{equation*}
Therefore, using compactness of the support of $f$, we get that
\begin{equation*}
I_2\leq \epsilon \|f\|_{L^1(\R^n)}.
\end{equation*}
Hence, for all $j\geq j_0$
\begin{equation*}
\left|\int_{\R^n}f(x)\:d\mu_j(x)-\int_{\R^n}f(x)\:d\mu(x)\right|\leq\epsilon (2A+\|f\|_{L^1(\R^n)}).
\end{equation*}
This proves the result.
\end{proof}
We will also need the following measure theoretic result proved in \cite[Proposition 2.6]{UR}.
\begin{lem}\label{mth}
Suppose $\{\mu_j\}$, $\mu$ are positive measures on $\R^n$ and $\{\mu_j\}$ converges to $\mu$ in weak*. Then for some $L\in[0,\infty)$, $\mu=Lm$ if and only if $\{\mu_j(B)\}$ converges to $Lm(B)$ for every open ball $B\subset\R^n$.
\end{lem}
\begin{rem}
The above lemma was proved in \cite{UR} under certain restriction on $\{\mu_j\}$ and $\mu$. But it is not hard to see that the same proof works as well under the hypothesis of the lemma above.
\end{rem}
We shall next state a result regarding comparison of the Hardy-Littlewood maximal function and the heat maximal function. This result is well known to the experts but since we could not find any reference of this result in the form in which it will be needed, we include a proof of it in the following.
We recall that for a positive measure $\mu$, its Hardy-Littlewood maximal function $M_{HL}(\mu)$ is defined by
\begin{equation*}
M_{HL}(\mu)(x_0)=\sup_{r>0}\frac{\mu(B(x_0,r))}{m(B(x_0,r))},\:\:\:\:\:\: x_0\in\R^n.
\end{equation*}
\begin{lem}\label{maximal}
If $\mu\in M$ and $\alpha>0$, then there exist positive constants $c_{\alpha}$ and $c_n$ such that
\begin{equation}
c_nM_{HL}(\mu)(x_0)\leq\sup_{t>0}W\mu(x_0,t^2)\leq\sup_{(x,t)\in \texttt{P}(x_0,\alpha)}W\mu(x,t)\leq c_{\alpha}M_{HL}(\mu)(x_0),
\end{equation}
for all $x_0\in\R^n$. The constants $c_n$ and $c_{\alpha}$ are independent of $x_0$.
\end{lem}
\begin{proof}
The second inequality is trivial as the set $\{(x_0,t^2)\mid t>0\}$ is contained in $\texttt{P}(x_0,\alpha)$ for any $\alpha>0$. The proof of the third inequality can be found in \cite[P.137]{Sa}. The first inequality is also easy, as for any $t>0$, \begin{eqnarray*}
W\mu(x_0,t^2)&=&(4\pi t^2)^{-\frac{n}{2}}\int_{\R^n}e^{-\frac{\|x_0-y\|^2}{4t^2}}\:d\mu(y)\\
&\geq & (4\pi )^{-\frac{n}{2}}t^{-n}\int_{B(x_0,2t)}e^{-\frac{\|x_0-y\|^2}{4t^2}}\:d\mu(y)\\
&\geq & (4\pi )^{-\frac{n}{2}}t^{-n}\int_{B(x_0,2t)}e^{-1}\:d\mu(y)\\
&=&e^{-1}m(B(0,1))2^n(4\pi )^{-\frac{n}{2}}\frac{\mu(B(x_0,2t))}{m(B(x_0,2t))}.
\end{eqnarray*}
Taking supremum over $t$ on both sides we get
\begin{equation}\label{radial}
c_nM_{HL}(\mu)(x_0)\leq\sup_{t>0}W\mu(x_0,t^2),
\end{equation}
where $c_n=e^{-1}m(B(0,1))2^n(4\pi )^{-\frac{n}{2}}$.
\end{proof}
To prove our main result we will also need an analogue of Montel's theorem for solutions of the heat equation. Using the well-known fact that the heat operator $\frac{\partial}{\partial t}-\Delta$ is hypoelliptic on $\R^{n+1}_+$ one can get a Montel-type result for solutions of the heat equation from a very general theorem proved in \cite[Theorem 4]{B}.
\begin{lem}\label{montel}
Let $\{u_j\}$ be sequence of solutions of the heat equation in $\R^{n+1}_+$. If $\{u_j\}$ is locally bounded then it has a subsequence which converges normally to a function $v$ (defined on $\R^{n+1}_+)$ which is also a solution of the heat equation.
\end{lem}
Like positive harmonic functions in $\R^{n+1}_+$ (they are given by convolution of measures with the Poisson kernel), we have similar characterization of the positive solutions of the heat equation.
\begin{lem}(\cite{W1}, P.93-99)\label{widder}
Let $u$ be a positive solution of the heat equation on $\R^{n+1}_+$. Then there exists a unique positive measure $\mu$ on $\R^n$ such that $u=W\mu$. In this case, we say that $\mu$ is the boundary measure of $u$.
\end{lem}
Given a function $F$ on $\R^{n+1}_+$ and $r>0$, we consider a nonisotropic dilation of $F$ given by
\begin{equation}\label{dilatedf}
F_r(x,t)=F(rx,r^2t),\:(x,t)\in \R^{n+1}_+.
\end{equation}
\begin{rem}\label{dilatef}
This notion of non-isotropic dilation is crucial for us primarily because of the following reasons.
\begin{enumerate}
\item[i)] If $F\in C^2(\R^{n+1}_+)$ is a solution of the heat equation then so is $F_r$
for every $r>0$. This follows easily by standard differentiation rules.
\item[ii)]
$(x,t)\in \texttt{P}(0,\alpha)$ if and only if $(rx,r^2t)\in \texttt{P}(0,\alpha)$ for every $r>0$.
\end{enumerate}
\end{rem}
Given $\nu\in M$ and $r>0$, we also define the dilate $\nu_r$ of $\nu$ by
\begin{equation}\label{dilatem}
\nu_r(E)=r^{-n}\nu(rE),
\end{equation}
for every Borel set $E\subset\R^n$.
We now prove a simple lemma involving the above dilates.which will be used in the proof of our main result.
\begin{lem}\label{dilate}
If $\nu\in M$, then for every $r>0$,  $W(\nu_r)=(W\nu)_r$.\end{lem}
\begin{proof}
For $E\subset\R^n$ a Borel set, using (\ref{dilatem}) it follows that
\begin{equation*}
\int_{\R^n}\chi_E\:d\nu_r=r^{-n}\nu(rE)=r^{-n}\int_{\R^n}\chi_{rE}(x)\:d\nu(x)=r^{-n}\int_{\R^n}\chi_{E}\left(\frac{x}{r}\right)\:d\nu(x).
\end{equation*}
Hence, for all nonnegative measurable functions $f$ we have  \begin{equation*}
\int_{\R^n}f(x)\:d\nu_r(x)=r^{-n}\int_{\R^n}f\left(\frac{x}{r}\right)\:d\nu(x).
\end{equation*}
It now follows from the relation above that for all $(x,t)\in\R^{n+1}_+$,
\begin{eqnarray*}
W(\nu_r)(x,t)&=& \int_{\R^n}W(x-y,t)\:d\nu_r(y)\\
&=&r^{-n}\int_{\R^n}W\left(x-\frac{y}{r},t\right)\:d\nu(y)\\
&=&r^{-n}(4\pi t)^{-\frac{n}{2}}\int_{\R^n}e^{-\frac{\|x-\frac{y}{r}\|^2}{4t}}\:d\nu(y)\\
&=&(4\pi tr^2)^{-\frac{n}{2}}\int_{\R^n}e^{-\frac{\|rx-y\|^2}{4tr^2}}\:d\nu(y)\\
&=&(W\nu)_r(x,t).
\end{eqnarray*}
\end{proof}
We end this section with an uniqueness theorem for solutions of the heat equation on $\R^{n+1}_+$. For a solution $u$ of the heat equation on $\R^{n+1}_+$, the real analyticity of $u(\cdot ,t)$ is usually
true for generic solutions of the heat equation but real analyticity of $u(x,\cdot)$ is false in general.  Nevertheless, the following uniqueness result holds.
\begin{lem}(\cite{P}, Theorem 1.2)\label{poon}
Let $u$ be a solution of the heat equation in $\R^{n+1}_+$. If $u$ vanishes of infinite order in space-time at a point $(x_0,t_0)\in \R^{n+1}_+$, then $u$ is identically zero.
\end{lem}
Here, vanishing of infinite order in space-time at a point $(x_0,t_0)$ means
that there exist a positive constant $C$ and an open neighbourhood $V$ of $(x_0,t_0)$ such that
\begin{equation*}
|u(x,t)|\leq C(\|x-x_0\|+|t-t_0|)^k,
\end{equation*}
for all $k\in\N$ and for all $(x,t)\in V$.
For a related result see \cite[Corollary 5]{KR}. In particular, it follows from Lemma \ref{poon} that if $u$ vanishes on a nonempty open subset of $\R^{n+1}_+$, then it vanishes identically.

\section{main result}
We shall first prove a special case of our main result. The proof of the main result will follow by reducing matters to this special case.
\begin{thm}\label{specialth}
Suppose $u$ is a positive solution of the heat equation on $\R^{n+1}_+$ and $L\in[0,\infty)$. If the boundary measure $\mu$ of $u$ is finite then the following statements hold.
\begin{enumerate}
\item[i)]If there exists $\eta>0$, such that
\be\label{etapara}
\lim_{\substack{(x,t)\to(0,0)\\(x,t)\in \texttt{P}(0,\eta)}}u(x,t)=L,
\ee
then the strong derivative of $\mu$ at zero is also equal to $L$.
\item[ii)]If the strong derivative of $\mu$ at zero is equal to $L$ then $u$ has parabolic limit $L$ at zero.
\end{enumerate}
\end{thm}
\begin{proof}
We first prove $i)$. We choose an open ball  $B_0\subset\R^n$, a sequence of positive real numbers $\{r_j\}$ converging to zero and consider the quotient
\begin{equation*}
M_j= \frac{\mu(r_jB_0)}{m(r_jB_0)}.
\end{equation*}
Assuming (\ref{etapara}), we will prove that $\{M_j\}$ is a bounded sequence and every convergent subsequence of $\{M_j\}$ converges to $L$. We first choose a positive real number $s$ such that $B_0$ is contained in $B(0,s)$. Then
\begin{equation}\label{mj}
M_j\leq \frac{\mu(r_jB(0,s))}{m(r_jB_0)}=\frac{\mu(r_jB(0,s))}{m(r_jB(0,s))}\times\frac{m(B(0,s))}{m(B_0)}\leq \frac{m(B(0,s))}{m(B_0)}M_{HL}(\mu)(0).
\end{equation}
Since $\mu$ is the boundary measure for $u$ we have that
\begin{equation*}
u(x,t)=W\mu (x,t),\:\:\:\:\:\text{for all $(x,t)\in\R^{n+1}_+$.}
\end{equation*}
By hypothesis, $u(0,t^2)$ converges to $L$ as $t$ tends to zero which implies, in particular, that there exists a positive number $\beta$ such that
\begin{equation*}
\sup_{t<\beta}u(0,t^2)<\infty.
\end{equation*}
Since $\mu$ is a finite measure we also have
\begin{equation*}
W\mu(0,t^2)\leq(4\pi t^2)^{-\frac{n}{2}}\int_{\R^n}d\mu,
\end{equation*}
and hence
\begin{equation}\label{tbig1}
\sup_{t\geq \beta}u(0,t^2)=\sup_{t\geq \beta}W\mu(0,t^2)\leq\sup_{t\geq \beta}(4\pi t^2)^{-\frac{n}{2}}\mu(\R^n)\leq(4\pi\beta^2)^{-\frac{n}{2}}\mu(\R^n)<\infty.
\end{equation}
It follows that $u(0,t^2)$ is a bounded function of $t\in (0,\infty)$. Lemma \ref{maximal}, now implies that $M_{HL}(\mu)(0)$ is finite. Boundedness of the sequence $\{M_j\}$ is now a consequence of the inequality (\ref{mj}).

It remains to prove that every convergent subsequence of $\{M_j\}$ converges to $L$. We choose a convergent subsequence of $\{M_j\}$, for the sake of simplicity, again denoted by $\{M_j\}$. For $j\in\N$, we define
\begin{equation*}
u_j(x,t)=u(r_jx,r_j^2t),\:\:\:\:(x,t)\in \R^{n+1}_+.
\end{equation*}
Then by Remark \ref{dilatef}, i), $\{u_j\}$ is a sequence of solutions of the heat equation in $\R^{n+1}_+$. We claim that $\{u_j\}$ is locally bounded. To see this, choose any compact set  $K\subset\R^{n+1}_+$. Then there exists a positive number $\alpha$ such that $K$ is contained in the parabolic region $\texttt{P}(0,\alpha)$. Indeed, we consider the map
\begin{equation*}
(x,t)\mapsto\frac{\|x\|^2}{t},\:\:\:\:(x,t)\in \R^{n+1}_+.
\end{equation*}
Clearly, this map is continuous. As $K$ is compact, image of $K$ under this map is bounded
and hence there exists a positive real number $\alpha$ such that \begin{equation*}
\frac{\|x\|^2}{t}<\alpha,\:\:\:\:\text{for all $(x,t)\in K$.}
\end{equation*}
Using the invariance of $\texttt{P}(0,\alpha)$ under nonisotropic dilation (see Remark \ref{dilatef}, ii)) and Lemma \ref{maximal}, it follows that for all $j\in\N$
\begin{equation*}
\sup_{(x,t)\in \texttt{P}(0,\alpha)}u_j(x,t)\leq \sup_{(x,t)\in \texttt{P}(0,\alpha)}u(x,t)\leq c_{\alpha}M_{HL}(\mu)(0).
\end{equation*}
Hence, $\{u_j\}$ is locally bounded.
Lemma \ref{montel} now guarantees the existence of a subsequence $\{u_{j_k}\}$ of $\{u_j\}$ which converges normally to a positive solution $v$ of the heat equation in $\R^{n+1}_+$.
We claim that for all $(x,t)\in\R^{n+1}_+$
\begin{equation}\label{vlimit}
v(x,t)=L=W(Lm)(x,t).
\end{equation}
To see this, choose any $(x_0,t_0)\in \texttt{P}(0,\eta)$. Since $\{r_{j_k}\}$ converges to zero as $k$ goes to infinity and $u(x,t)$ has limit $L$, as $(x,t)$ tends to $(0,0)$ within $\texttt{P}(0,\eta)$,
\begin{equation*}
v(x_0,t_0)=\lim_{k\to\infty}u_{j_k}(x_0,t_0)=\lim_{k\to\infty}u(r_{j_k}x_0,r_{j_k}^2t_0)=L,
\end{equation*}
as $(r_{j_k}x_0,r_{j_k}^2t_0)\in \texttt{P}(0,\eta)$ for all $j_k\in\N$. It is now immediate from the uniqueness result (see the discussion following Lemma \ref{poon}) that $v$ is the constant function $L$ which settles the claim.
On the other hand, by Lemma \ref{dilate}
\begin{equation}\label{udilate}
u_{j_k}(x,t)=u(r_{j_k}x,r_{j_k}^2t)=(W\mu)(r_{j_k}x,r_{j_k}^2t)=W(\mu_{r_{j_k}})(x,t).
\end{equation}
It follows from (\ref{vlimit}) and (\ref{udilate}) that $\{W(\mu_{r_{j_k}})\}$ converges normally to $W(Lm)$. It follows from Lemma \ref{normal} that the sequence of measures $\{\mu_{r_{j_k}}\}$ converges to $Lm$ in weak* and hence by Lemma \ref{mth}, $\{\mu_{r_{j_k}}(B)\}$ converges to $Lm(B)$ for every  ball $B\subset\R^n$.
Therefore,
\begin{equation*}
Lm(B_0)=\lim_{k\to \infty}\mu_{r_{j_k}}(B_0)=\lim_{k\to\infty}{r_{j_k}}^{-n}\mu({r_{j_k}}B_0)=\lim_{k\to\infty}\frac{\mu({r_{j_k}}B_0)}{m({r_{j_k}}B_0)}m(B_0)=m(B_0)\lim_{k\to\infty}M_{j_{k}}.
\end{equation*}
This implies that the sequence $\{M_{j_{k}}\}$ converges to $L$ and hence, so does $\{M_j\}$. This completes the proof of $i)$.

Now, we prove $ii)$. We suppose that the strong derivative of $\mu$ at zero is equal to $L$ but the parabolic limit of $u$ at zero is not equal to $L$. Then there exists a positive number $\alpha$ and a sequence $(x_j,t_j^2)\in \texttt{P}(0,\alpha)$ with $(x_j,t_j^2)$ converging to $(0,0)$ but $\{u(x_j,t_j^2)\}$ fails to converge to $L$. Since the limit defining $D\mu(0)$ exists finitely (in fact, equal to $L$) it follows, in particular, that
\begin{equation*}
\sup_{0<r<1}\frac{\mu(B(0,r))}{m(B(0,r))}<\infty.
\end{equation*}
Finiteness of the measure $\mu$ and Lemma \ref{maximal} then implies that
\begin{equation*}
\sup_{(x,t)\in \texttt{P}(0,\alpha)}u(x,t)\leq c_{\alpha} M_{HL}(\mu)(0)<\infty.
\end{equation*}
This shows that $\{u(x_j,t_j^2)\}$ is a bounded sequence. We then consider a convergent subsequence of this sequence, again denoted by $\{u(x_j,t_j^2)\}$, for the sake of simplicity, such that
\begin{equation}\label{limitL}
\lim_{j\to\infty} u(x_j,t_j^2)=L'.
\end{equation}
We will prove that $L'$ is equal to $L$. Using the sequence $\{t_j\}$, we consider the dilates
\begin{equation*}
u_j(x,t)=u(t_jx,t_j^2t),\:\:\:\:(x,t)\in \R^{n+1}_+.
\end{equation*}
Arguments used in the first part of the proof shows that $\{u_j\}$ is a locally bounded sequence of positive solutions of the heat equation in $\R^{n+1}_+$. Hence, by Lemma \ref{montel}, there exists a subsequence $\{u_{j_k}\}$ of $\{u_j\}$ which converges normally to a positive solution $v$ of the heat equation in $\R^{n+1}_+$. Lemma \ref{widder} therefore shows that there exists  $\nu\in M$ such that $v$ equals $W\nu$. We now consider the sequence of dilates $\{\mu_k\}$ of $\mu$ by $\{t_{j_k}\}$ according to (\ref{dilatem}). An application of Lemma \ref{dilate} then implies that $W\mu_k=u_{j_k}$. It follows that the sequence of functions $\{W\mu_k\}$ converges normally to $W\nu$. By Lemma \ref{normal}, we thus obtain weak* convergence of $\{\mu_k\}$ to $\nu$.

Since $D\mu(0)=L$, it follows that for any open ball $B\subset\R^n$,
\begin{equation*}
\lim_{k\to \infty}\mu_k(B)=\lim_{k\to\infty}{t_{j_k}}^{-n}\mu({t_{j_k}}B)=\lim_{k\to\infty}\frac{\mu({t_{j_k}}B)}{m({t_{j_k}}B)}m(B)=Lm(B).
\end{equation*}
Hence by Lemma \ref{mth}, $\nu=Lm$. As $v=W\nu$ it follows that
\begin{equation*}
v(x,t)=L,\:\:\:\:\text{for all $(x,t)\in \R^{n+1}_+$.}
\end{equation*}
This, in turn, implies that $\{u_{j_k}\}$ converges to the constant function $L$ normally in $\R^{n+1}_+$. On the other hand, we note that
\begin{equation*}
u(x_{j_k},t_{j_k}^2)=u\left(t_{j_k}\frac{x_{j_k}}{t_{j_k}},t_{j_k}^2\right)=u_{j_k}\left(\frac{x_{j_k}}{t_{j_k}},1\right).
\end{equation*}
Since $(x_{j_k},t_{j_k}^2)$ belongs to the parabolic region $\texttt{P}(0,\alpha)$, for all $k\in\N$, it follows that
\begin{equation*}
\left(\frac{x_{j_k}}{t_{j_k}},1\right)\in\overline B(0,\sqrt{\alpha})\times\{1\},
\end{equation*}
which is a compact subset of $\R^{n+1}_+$. Therefore,
\begin{equation*}
\lim_{k\to\infty}u(x_{j_k},t_{j_k}^2)=L.
\end{equation*}
In view of (\ref{limitL}), we can thus conclude that $L'$ equals $L$. So, every convergent subsequence of the original sequence $\{u(x_j,t_j^2)\}$ converges to $L$. This contradicts our assumption that $\{u(x_j,t_j^2)\}$ fails to converge to $L$. This completes the proof.
\end{proof}
Now, we are in a position to state and prove our main result.
\begin{thm}\label{main}
Suppose $u$ is a positive solution of the heat equation on $\R^{n+1}_+$ and suppose $x_0\in\R^n$, $L\in[0,\infty)$. If $\mu$ is the boundary measure of $u$ then the following statements hold.
\begin{enumerate}
\item[i)]If there exists $\eta>0$ such that
\bes
\lim_{\substack{(x,t)\to(x_0,0)\\(x,t)\in \texttt{P}(x_0,\eta)}}u(x,t)=L,
\ees
then the strong derivative of $\mu$ at $x_0$ is also equal to $L$.
\item[ii)]If the strong derivative of $\mu$ at $x_0$ is equal to $L$ then $u$ has parabolic limit $L$ at $x_0$.
\end{enumerate}
\end{thm}
\begin{proof}
We consider the translated measure $\mu_0=\tau_{-x_0}\mu$, where
\begin{equation*}
\tau_{-x_0}\mu (E)=\mu(E+x_0),
\end{equation*}
for all Borel subsets $E\subset \R^n$. Using translation invariance of the Lebesgue measure, it follows from the definition of strong derivative that $D\mu_0(0)$ and $D\mu (x_0)$ are equal. Since $W\mu_0$ is given by the convolution of $\mu_0$ with $h_{\sqrt{t}}$ and translation commutes with convolution, it follows that
\begin{equation}\label{trans}
W\mu_0(x,t)= (h_{\sqrt{t}}\ast \tau_{-x_0}\mu )(x)=\tau_{-x_0}(h_{\sqrt{t}}\ast\mu )(x)=W\mu(x+x_0,t).
\end{equation}
We fix an arbitrary positive number $\alpha$. As $(x,t)\in \texttt{P}(0,\alpha)$ if and only if $(x_0+x,t)\in \texttt{P}(x_0,\alpha)$, one infers from  (\ref{trans}) that
\begin{equation*}
\lim_{\substack{(x,t)\to(0,0)\\(x,t)\in \texttt{P}(0,\alpha)}}W\mu_0(x,t)=\lim_{\substack{(\xi,t)\to(x_0,0)\\(\xi,t)\in \texttt{P}(x_0,\alpha)}}W\mu(\xi,t).
\end{equation*}
Hence, it suffices to prove the theorem under the assumption that $x_0$ is the origin. We now show that we can even take $\mu$ to be a finite measure. Let $\tilde{\mu}$ be the restriction of $\mu$ on the unit ball $B(0,1)$. If $B(y,s)$ is any given open ball then for all positive $r$ smaller than $(s+\|y\|)^{-1}$, it follows  that $rB(y,s)$ is a subset of $B(0,1)$. This in turn implies that the quantities $D\mu(0)$ and $D\tilde{\mu}(0)$ are equal.
We now claim that
\begin{equation}\label{finaleq}
\lim_{\substack{(x,t)\to(0,0)\\(x,t)\in \texttt{P}(0,\alpha)}}W\mu(x,t)=\lim_{\substack{(x,t)\to(0,0)\\(x,t)\in \texttt{P}(0,\alpha)}}W\tilde{\mu}(x,t).
\end{equation}
Once we prove this claim, the theorem follows straight away from Theorem \ref{specialth} as $\alpha$ is arbitrary. In this regard, we first observe that
\begin{equation*}
\lim_{t\to 0}\int_{\R^n\setminus B(0,1)}W(x-y,t)\:d\mu (y)=0,
\end{equation*}
uniformly for $x\in B(0,1/2)$. Indeed, using the expression of $W(x,t)$, and the elementary inequalities
\begin{equation*}
\|x-y\|>\frac{1}{2},\:\:\:\:\:\:\:\|x\|<\frac{\|y\|}{2},
\end{equation*}
it follows for all $x\in B(0,1/2)$ that 
\begin{equation*}
\int_{\R^n\setminus B(0,1)}W(x-y,t)\:d\mu (y)\leq (4\pi t)^{-\frac{n}{2}}e^{-\frac{1}{32t}}\int_{\R^n\setminus B(0,1)}e^{-\frac{\|y\|^2}{32}}\:d\mu (y),\:\:\:\:\:\:\text{for $t\leq 1$.}
\end{equation*}
As the latter expression goes to zero as $t$ tends to zero, the observation follows. Now,
\begin{eqnarray*}
W\mu(x,t)&=&\int_{B(0,1)}W(x-y,t)\:d\mu (y)+\int_{\R^n\setminus B(0,1)}W(x-y,t)\:d\mu (y)\\
&=& W\tilde{\mu}(x,t)+ \int_{\R^n\setminus B(0,1)}W(x-y,t)\:d\mu (y).
\end{eqnarray*}
Given any positive $\epsilon$, we get some $t_0>0$, such that for all $t\in (0,t_0)$, the integral on the right hand side of the equation above is smaller than $\epsilon$ for all $x\in B(0,1/2)$. On the other hand, if we choose $t\in (0,1/4\alpha)$ then it is immediate that
\begin{equation*}
\texttt{P}(0,\alpha)\cap \{(x,t)\in\R^{n+1}_+\mid t\in (0,1/4\alpha ) \}\subset B(0,1/2)\times (0,1/4\alpha).
\end{equation*}
Hence, for all $(x,t)\in \texttt{P}(0,\alpha)$ with $t$ smaller that $\text{min } \{t_0,1/4\alpha\}$ we have
\begin{equation*}
|W\mu (x,t)-W\tilde{\mu}(x,t)|<\epsilon.
\end{equation*}
This proves (\ref{finaleq}) and hence the theorem.
\end{proof}
As an immediate consequence of the theorem above we have the following.
\begin{cor}
Suppose $u$ is a positive solution of the heat equation on $\R^{n+1}_+$. If there exists $x_0\in\R^n$ and $L\in[0,\infty)$, such that for some $\eta>0$
\begin{equation*}
\lim_{\substack{(x,t)\to(x_0,0)\\(x,t)\in \texttt{P}(x_0,\eta)}}u(x,t)=L,
\end{equation*}
then for every $\alpha>0$
\begin{equation*}
\lim_{\substack{(x,t)\to(x_0,0)\\(x,t)\in \texttt{P}(x_0,\alpha)}}u(x,t)=L.
\end{equation*}
\end{cor}

\section*{acknowledgements}
The author would like to thank Swagato K. Ray for suggesting this problem and for many
useful discussions during the course of this work. The author would also like to thank Y. Kannai for several useful discussions.

\end{document}